\documentclass{amsart}
\usepackage{latexsym, amssymb}
\usepackage{amsthm}

\def\beq{\begin{equation}}
\def\eeq{\end{equation}}\usepackage{url}

\newtheorem{lem}{Lemma}
\newtheorem{thm}[lem]{Theorem}
\newtheorem{prp}[lem]{Proposition}
\newtheorem{cor}[lem]{Corollary}

\theoremstyle{definition}

\newtheorem{rem}{Remark}

\def\ra{\rightarrow}

\def\beqa{\begin{eqnarray}}
\def\eeqa{\end{eqnarray}}
\def\beqa{\begin{eqnarray}}
\def\eeqa{\end{eqnarray}}

\def\Out{\mathrm{Out}}

\def\vol{\mathrm{vol}}

\def\Isom{\mathrm{Isom}}

\begin{document}
\title[Periodic flats and group actions on locally symmetric spaces]{Periodic flats and group 
actions on locally symmetric spaces}

\author{Grigori Avramidi}
\address{Dept. of Mathematics\\
5734 S. University Avenue\\
Chicago, Illinois 60637}
\email[G.~Avramidi]{gavramid@math.uchicago.edu}




\begin{abstract}
We use maximal periodic flats to show that on a finite volume
irreducible locally symmetric manifold of dimension $\geq 3$, no metric 
has more symmetry than the locally symmetric metric.
We also show that if a finite volume metric is not locally symmetric, 
then its lift to the universal cover has discrete isometry group. 
\end{abstract}
\maketitle
\section{Introduction}
The goal of this paper is to give some new results of the flavor that 
`the locally symmetric metric on a locally symmetric manifold is the most symmetric
metric on that manifold'. Earlier work, reviewed below, establishes such results for 
compact manifolds. Our goal here is to dispense with compactness, where new obstacles and
phenomena arise. 
Let $(M,g)$ be a complete Riemannian manifold homeomorphic to a finite volume, 
nonpositively curved locally symmetric space with no local torus factors.
In this paper we prove that the metric $g$ has no homotopically trivial isometries. 
More precisely, the isometry group $\Isom(M,g)$ acts on free homotopy classes of loops and this
gives an action homomorphism 
\begin{equation}
\rho:\Isom(M,g)\ra\Out(\pi_1M).
\end{equation}

\begin{thm}
\label{maintheorem}
The homomorphism $\rho$ is injective.
\end{thm}
If, in addition, the locally symmetric space is irreducible and of dimension $\geq 3$, then Margulis-Mostow-Prasad rigidity
shows that the group $\Out(\pi_1M)$ is represented by isometries of the locally symmetric metric. 
Thus, $\Isom(M,g)<\Isom(M,h_{sym})$ so that, in a sense, the locally symmetric metric has the most symmetry. 
If the metric $g$ has finite volume, then we also get the following dichotomy for the isometry group of the universal
cover.
\begin{thm}
\label{fw}
Let $(M,h_{sym})$ be a finite volume, irreducible, locally symmetric manifold of dimension $\geq 3.$
Suppose that $M$ has no local torus factors. 
If $g$ is any complete, finite volume Riemannian metric on $M$ then either,
\begin{itemize}
\item
$g$ is a constant multiple of the locally symmetric metric, or
\item
the isometry group of the universal cover $\Isom(\widetilde M,\widetilde g)$ is discrete and contains
$\pi_1M$ as a subgroup of index $\leq\mathrm{vol}(M,h_{sym})/\varepsilon(h_{sym}).$
\end{itemize} 
\end{thm}
The constant $\varepsilon(h_{sym})$ is the volume of the smallest locally symmetric orbifold covered by $(M,h_{sym}).$
It does not depend on the metric $g$. The assumption that $g$ has finite volume turns out to be necessary.
In fact, one can construct a complete infinite volume metric on $M$ such that the isometry
group of the universal cover contains $\pi_1M*\mathbb Z$ as a discrete subgroup (see 
section $12$ of \cite{avramidi}.)

\subsection*{Related work}
For closed locally symmetric spaces, Theorem \ref{maintheorem} was proved by Borel (published by Conner
and Raymond, Theorem 3.2 of \cite{connerraymond})
and Theorem \ref{fw} was proved by Farb and Weinberger (Theorem 1.7 of \cite{farbweinbergerisometries}).
For finite volume locally symmetric spaces of non-zero Euler characteristic, Theorem 2 is proved
(using $L^2$ cohomology rather than maximal periodic flats) in \cite{avramidi}. 
For noncompact locally symmetric spaces with vanishing Euler characteristic, both theorems are new. 

One also has results similar to Theorem \ref{fw} for aspherical manifolds (or orbifolds) which are not locally symmetric.
The analogous result for moduli spaces of algebraic curves---Royden's theorem---is treated in 
\cite{farbweinbergerroyden} and \cite{avramidi}. For a class of closed aspherical manifolds which are tiled by 
locally symmetric spaces (piecewise locally symmetric
spaces) Theorem \ref{fw} is proved by T$\hat{\mathrm{a}}$m Nguy$\tilde{\hat{\mathrm{e}}}$n Phan 
in \cite{tamsymmetry}. For more general aspherical Riemannian and Lorentz manifolds,
one has results about the structure of isometry groups (\cite{farbweinbergerisometries} and \cite{melnicklorentz}),
but no quantitative bound (depending only on the topology of $M$) 
on how many isometries of the universal cover are not covering translations, i.e. on 
the size of $\Isom(\widetilde M,\widetilde g)/\pi_1M$ 
(see Conjecture 1.6 in \cite{farbweinbergerisometries}.)

\subsection*{Acknowledgements} I would like to thank my advisor Shmuel Weinberger
for kindling my interest in isometries of aspherical manifolds, and for 
his guidance. I would also like to thank Benson Farb for a number 
of inspiring conversations, and T$\hat{\mathrm{a}}$m Nguy$\tilde{\hat{\mathrm{e}}}$n Phan 
for drawing my attention to the paper \cite{pettetsouto}.

\section{\label{locallysymmetricborel}Maximal periodic flats in locally symmetric spaces}
Our proof of Theorem \ref{maintheorem}
is based on a result of Pettet and Souto \cite{pettetsouto}, which says that the fundamental group
of a locally symmetric space contains a free abelian subgroup that `moves the boundary as much
as possible'. We recall it now. 

Suppose that $G$ is a connected semisimple Lie group 
with no compact or Euclidean factors.
Let $K<G$ be a maximal compact and $\Gamma<G$ a torsionfree lattice.
Suppose that $r:=\mathrm{rank}_{\mathbb R}G$ is the real rank of $G$. 
By a {\it maximal $\Gamma$-periodic flat} we will mean a totally geodesic $r$-dimensional 
flat subspace $\mathbb R^r\subset G/K$ which is invariant 
under some free abelian subgroup $\mathbb Z^r<\Gamma.$\footnote{The image of $\mathbb R^r$ under $G/K\ra\Gamma\setminus G/K$
is what \cite{pettetsouto} call a {\it maximal periodic flat}.}  

We will use the following facts. For an arithmetic locally symmetric space, these are standard (see \cite{borelserre}).
For general locally symmetric spaces, these facts are described in \cite{pettetsouto}.
The unpublished manuscript \cite{morris} is a general reference on arithmetic groups.

\begin{enumerate}
\item
The symmetric space $G/K$ has a maximal $\Gamma$-periodic flat.
\item
There is a number $q=\mathrm{rank}_{\mathbb Q}(M)$ called the $\mathbb Q$-rank of the 
locally symmetric space $M:=\Gamma\setminus G/K$ and a
$(q-1)$-dimensional simplicial complex $\Delta_{\mathbb Q}(M)$ on which the 
group $\Gamma$ acts simplicially. The complex is called the {\it rational Tits building}
of the locally symmetric space $M$. It is homotopy equivalent to an infinite
wedge of $(q-1)$-spheres, i.e. $\Delta_{\mathbb Q}(M)\cong\vee S^{q-1}.$
\item
The locally symmetric space $M$ is the interior of a compact manifold $\overline M$ 
with boundary $\partial\overline M.$ There is a $\Gamma$-equivariant homotopy equivalence

\begin{equation*}
\partial\widetilde{\overline M}\ra\Delta_{\mathbb Q}(M).
\end{equation*}
\end{enumerate}
For the rest of this paper, we will use $\mathbb R^r$ to denote a fixed maximal $\Gamma$-periodic
flat and $\mathbb Z^r<\Gamma$ a free abelian subgroup of rank $r$ which leaves that flat invariant. 
Pettet and Souto prove the following
\begin{thm}[Proposition 5.5. in \cite{pettetsouto}]
\label{stabilizers}
The stabilizer $\mathrm{Stab}_{\mathbb Z^r}(\sigma_k)$ of a $k$-simplex in the building $\Delta_{\mathbb Q}(M)$
is a free abelian group of rank $\leq r-k-1.$
\end{thm}
\begin{rem}
Heuristically, the stabilizer spans a torus which is complementary to the $(k+1)$-dimensional $\mathbb Q$-split
torus corresponding to $\sigma_k,$ and together they span a torus of rank 
$\mathrm{rank}_{\mathbb Z}(\mathrm{Stab}_{\mathbb Z^r}(\sigma_k))+k+1\leq r$. 
\end{rem}
Using this and a theorem of McMullen ($6.1$ in \cite{pettetsouto}), Pettet and Souto 
show that a maximal periodic flat cannot be homotoped to the boundary. On the level of universal covers, their theorem is

\begin{thm}[1.2 in \cite{pettetsouto}]
\label{notperipheral}
With the notations above, let $\mathbb R^r\subset G/K=\widetilde{\overline M}$
be a maximal $\Gamma$-periodic flat invariant under a free abelian subgroup $\mathbb Z^r<\Gamma.$
Then, $\mathbb R^r$ cannot be $\mathbb Z^r$-equivariantly homotoped into the boundary $\partial\widetilde{\overline M}.$
\end{thm}
The following consequence, pointed out to us by T$\hat{\mathrm{a}}$m Nguy$\tilde{\hat{\mathrm{e}}}$n Phan, 
is the reason for our initial interest in 
maximal periodic flats. 

\begin{cor}
\label{compact}
If $(M,g)$ is a complete, Riemannian manifold homeomorphic to a finite
volume, aspherical locally symmetric space with no local torus factors, 
then the group of homotopically trivial isometries
$K:=\ker(\rho:\Isom(M,g)\ra\Out(\pi_1M))$ is a compact Lie group.
\end{cor}

\begin{proof}
The group $K$ is a Lie group by the Myers-Steenrod theorem \cite{myerssteenrod}.
Let $N\subset\overline M$ be a collar neighborhood of the boundary.
If the group $K$ is not compact, then there is a homotopically
trivial isometry $\phi$ which sends $\overline M\setminus N$
into the collar neighborhood of the boundary $N.$ This isometry defines a $\mathbb Z^r$-equivariant
homotopy of a maximal periodic flat into a neighborhood of the boundary, contradicting theorem \ref{notperipheral}. 
\end{proof}

We will deduce Theorem \ref{notperipheral} from Theorem \ref{stabilizers} using equivariant homology
instead of McMullen's theorem. This seems to be a simplification, and in any case we need 
it for our proof of Theorem \ref{maintheorem}.

First, we will restate Theorem \ref{notperipheral} in a slightly different form.
Notice that if the flat $\mathbb R^r$ can be $\mathbb Z^r$-equivariantly 
homotoped to the boundary, then we have a $\mathbb Z^r$-equivariant map 
$s:\mathbb R^r\ra\partial\widetilde{\overline M}\ra\Delta_{\mathbb Q}(M).$
Give the product $\mathbb R^r\times\Delta_{\mathbb Q}(M)$ the diagonal $\mathbb Z^r$-action.
By projecting onto the first factor, we get a bundle 
$(\mathbb R^r\times\Delta_{\mathbb Q}(M))/\mathbb Z^r\ra\mathbb T^r$
over the torus $\mathbb T^r$ with fibre $\Delta_{\mathbb Q}(M).$
The $\mathbb Z^r$-equivariant map 

\begin{equation*}
\mathbb R^r\stackrel{v\mapsto(v,v)}\longrightarrow
\mathbb R^r\times\mathbb R^r\stackrel{id\times s}\longrightarrow\mathbb R^r\times\Delta_{\mathbb Q}(M)\end{equation*}
gives a section of this bundle. Thus, theorem \ref{notperipheral} follows from 

\begin{thm}
\label{nosec}
Let $\mathbb Z^r<\Gamma$ be a free abelian subgroup leaving
invariant a maximal $\Gamma$-periodic flat $\mathbb R^r\subset \widetilde M.$ 
Then, the bundle 
\begin{equation}
\label{bundle}
\Delta_{\mathbb Q}(M)\ra {\mathbb R^r\times\Delta_{\mathbb Q}(M)\over\mathbb Z^r}\ra\mathbb T^r
\end{equation}
does not have a section.
\end{thm}
\begin{proof}
We will show that $(\mathbb R^r\times\Delta_{\mathbb Q}(M))/\mathbb Z^r$ has no 
homology in dimension $r,$ which implies that (\ref{bundle}) cannot have a section.
(If there is a section $s,$ then $s(\mathbb T^r)$ represents a non-trivial homology class.)
The double complex
 
\begin{equation}
E^0_{s,t}:=C_s(\mathbb R^r)\otimes_{\mathbb Z^r}C_t(\Delta_{\mathbb Q}(M))
\end{equation}
has the spectral sequence
\begin{equation}
E^1_{s,t}:=H_s(\mathbb Z^r;C_t(\Delta_{\mathbb Q}(M)))\implies H_{s+t}
\left({\mathbb R^r\times\Delta_{\mathbb Q}(M)\over\mathbb Z^r}\right)
\end{equation}
associated to it. We will use the fact that $\mathbb Z^r$ acts with small stabilizers on the 
rational Tits building $\Delta_{\mathbb Q}(M).$
Note that  

\begin{eqnarray*}
\label{stabilize}
E^1_{s,t}:&=&H_s(\mathbb Z^r;C_t(\Delta_{\mathbb Q}(M)))\\
&=&H_s(\mathbb Z^r;\bigoplus_{\sigma_t}\mathbb Z^r\cdot\sigma_t)\\
&=&\bigoplus_{\sigma_t}H_s(\mathbb Z^r;\mathbb Z^r\cdot\sigma_t)\\
&=&\bigoplus_{\sigma_t}H_s(\mathrm{Stab}_{\mathbb Z^r}(\sigma_t)),
\end{eqnarray*}
where the sum is taken over all $\mathbb Z^r$-orbits of $t$-simplices.
By the result of Pettet-Souto (Theorem \ref{stabilizers}), 
$\mathrm{Stab}_{\mathbb Z^r}(\sigma_t)$ is a free abelian group of rank $\leq r-t-1,$
so we find that $E^1_{s,t}=0$ for $s\geq r-t,$ i.e. for $s+t\geq r.$ Since this spectral sequence converges to the
homology of $(\mathbb R^r\times\Delta_{\mathbb Q}(M))/\mathbb Z^r$, we find that 

\begin{equation}
\label{zero}
H_k\left({\mathbb R^r\times\Delta_{\mathbb Q}(M)\over \mathbb Z^r}\right)=0 \hspace{1cm}\mbox {for } k\geq r.
\end{equation}
This proves Theorem \ref{nosec}.
\end{proof}

\subsection*{Homology of $\partial\widetilde{\overline M}/\mathbb Z^r$}
For future use, we rephrase the above computation in terms of the homology of the Borel-Serre boundary.
We have homotopy equivalences 
\begin{equation}
\label{hequiv}
\partial\widetilde{\overline M}/\mathbb Z^r\leftarrow
(\mathbb R^r\times\partial\widetilde{\overline M})/\mathbb Z^r\ra(\mathbb R^r\times\Delta_{\mathbb Q}(M))/\mathbb Z^r.
\end{equation}
The left map is a homotopy equivalence because it is the projection
map of a bundle with fibre $\mathbb R^r.$ The right map is the obvious 
homotopy equivalence obtained from the $\mathbb Z^r$-equivariant
homotopy equivalence $\partial\widetilde{\overline M}\ra\Delta_{\mathbb Q}(M).$
Thus, equation (\ref{zero}) shows that the homology of the Borel-Serre boundary
of the $\mathbb Z^r$ cover vanishes in dimensions $\geq r$, i.e. 
\begin{equation}
H_k(\partial\widetilde{\overline M}/\mathbb Z^r)=0\hspace{1cm}\mbox {for } k\geq r.
\end{equation} 
This equation, the long exact homology sequence
\begin{equation}
\cdots\ra H_*(\widetilde{\overline M}/\mathbb Z^r)\ra H_*(\widetilde{\overline M}/\mathbb Z^r,\partial\widetilde{\overline M}/\mathbb Z^r)
\ra H_{*-1}(\partial\widetilde{\overline M}/\mathbb Z^r)\ra\cdots
\end{equation}
and the fact that $\widetilde{\overline M}/\mathbb Z^r$ has no homology above
dimension $r$ (it is homotopy equivalent to the $r$-torus) 
implies that 
\begin{equation}
H_k(\widetilde{\overline M}/\mathbb Z^r,\partial\widetilde{\overline M}/\mathbb Z^r)=0\hspace{1cm}\mbox {for } k>r.
\end{equation}
Everything we've said so far is valid for homology with coefficients in $\mathbb Z,$ and also
with coefficients in the field $\mathbb F_p$ of $p$ elements. 
It is useful to rewrite this in terms of homology with local coefficients in the
$\mathbb F_p[\Gamma]$ module $\mathbb F_p[\Gamma/\mathbb Z^r]:$

\begin{eqnarray}
\label{one}
H_k(\partial\overline M;\mathbb F_p[\Gamma/\mathbb Z^r])&=&0\hspace{1cm}\mbox {for } k\geq r,\\
\label{two}
H_k(\overline M,\partial\overline M;\mathbb F_p[\Gamma/\mathbb Z^r])&=&0\hspace{1cm}\mbox {for } k>r.
\end{eqnarray}
\section{Homotopically trivial $\mathbb Z/p$-actions}
\subsection{Outline of proof of Theorem \ref{maintheorem}}
We have seen that the group $K=\ker(\rho)$ of homotopically
trivial isometries is a compact Lie group (Corollary \ref{compact}). 
Thus, to show this group is trivial we only need to check that there are no elements of prime order $p$.
In other words, we need to show that the locally symmetric space 
$M$ has no non-trivial, homotopically trivial $\mathbb Z/p$-actions.
We lift the $\mathbb Z/p$-action to the cover $\widetilde{\overline M}/\mathbb Z^r$, 
look at the fixed point set `near the boundary' of this cover, and use 
(\ref{one}) and (\ref{two}) to show the fixed point set is everything.

The $\mathbb Z/p$-action may not extend to the Borel-Serre
boundary of $M$, so we need to replace homology of the boundary (\ref{one}) 
and homology relative to the boundary (\ref{two}) 
by homology of the end and homology with closed supports, respectively. 
We recall these notions in the next two subsections.

 
\subsection{Homology with closed supports}
Given a $\Gamma$-cover $\widetilde X\ra X$, denote by $C_*^{cl}(X;V)$
the complex of chains with {\it closed} support on $X$ and coefficients in the $\mathbb F_p[\Gamma]$-module $V$.
More precisely, first define $C_*^{cl}(X;\mathbb F_p[\Gamma])$ as the complex of those chains on the cover $\widetilde X$
which for every compact set $K\subset \widetilde X$ meet only finitely many $\Gamma$-translates of $K.$ 
Then, define $C_*^{cl}(X;V):=C_*^{cl}(X;\mathbb F_p[\Gamma])\otimes_{\mathbb F_p[\Gamma]}V.$  

If $M$ is the interior of a compact manifold with boundary, $\partial M\times(0,\infty)$ is an
open neighborhood of the boundary, and $M_0:=M\setminus (\partial M\times(0,\infty))$ its complement,
then the relative homology of the pair $(M_0,\partial M_0)$ is isomorphic to homology with closed supports
on $M$ via 
\begin{eqnarray}
\label{relative}
H_*(M_0,\partial M_0;V)&\cong &H_*^{cl}(M;V),\\
(c,\partial c)&\mapsto& c\cup_{\partial c} \partial c\times[0,\infty). 
\end{eqnarray}
One way to see this is to note that the map is Poincare dual (V.9.2 in \cite{bredonsheaf}) to the
cohomology isomorphism $H^{m-*}(M;\mathcal O\otimes V)\cong H^{m-*}(M_0;\mathcal O\otimes V)$, where $\mathcal O$ is the orientation module. 
\subsection{Homology of the end}
The complex $C^e_*(X;V)$ of {\it chains on the end of $X$} is defined
to be the quotient

\begin{equation}
\label{formalboundary}
0\ra C_{*+1}(X;V)\ra C_{*+1}^{cl}(X;V)\ra C^e_*(X;V)\ra 0.
\end{equation}
Denote the homology of this complex by $H^e_*(X;V):=H_*(C^e_*(X;V)).$
Associated to the short exact sequence of chain complexes (\ref{formalboundary}) there is a long exact homology sequence
\begin{equation}
\label{les}
\cdots\ra H^{cl}_{*+1}(X;V)\ra H^e_{*}(X;V)\ra H_{*}(X;V)\ra\cdots.
\end{equation}

Putting together the long exact homology sequence of the pair $(M_0,\partial M_0)$ with the long exact homology
sequence (\ref{les}), we get a commutative diagram
\begin{equation*}
\begin{array}{ccccccccc}
\cdots&\ra &H_{*+1}(M_0,\partial M_0;V)&\ra &H_{*}(\partial M_0;V)&\ra &H_{*}(M_0;V)&\ra&\cdots\\
      &    &\downarrow   &    &\downarrow                            &    &\downarrow         &   &\\
\cdots&\ra &H_{*+1}^{cl}(M;V)&\ra &H^e_{*}(M;V)         &\ra &H_{*}(M;V)       &\ra&\cdots.
\end{array}
\end{equation*}
The left vertical arrow is an isomorphism by (\ref{relative}) and the 
right vertical arrow is an isomorphism because (ordinary) homology is homotopy invariant, so 
the middle arrow is an isomorphism 

\begin{eqnarray}
\label{boundary}
H_*(\partial M_0;V)&\cong&H^e_*(M;V),\\
\nonumber
a&\mapsto& a\times[0,\infty).
\end{eqnarray}
Putting together (\ref{one}), (\ref{two}), (\ref{relative}) and (\ref{boundary}) we get  
\begin{eqnarray}
\label{end}
H_k^e(M;\mathbb F_p[\Gamma/\mathbb Z^r])=0 \mbox{ for } k\geq r,\\
\label{closedzero}
H_k^{cl}(M;\mathbb F_p[\Gamma/\mathbb Z^r])=0 \mbox{ for } k>r.
\end{eqnarray}

\subsection{Proof of Theorem \ref{maintheorem}}
Let $F\subset M$ be the fixed point set of a homotopically trivial $\mathbb Z/p$-action.
Let $f$ be the dimension of $F$ and let $m$ be the dimension of $M$. Our goal is to show that 
$f=m$, i.e. that the fixed point set has the same dimension as the manifold $M$. Since the fixed point
set is a closed submanifold of $M$, this will show that it is everything, i.e. that the $\mathbb Z/p$-action is
trivial.

The inclusion $F\hookrightarrow M$ induces maps on homology, 
homology with closed supports and homology of the end with coefficients in $V:=\mathbb F_p[\Gamma/\mathbb Z^r]$.
These fit together in the commutative diagram
\begin{equation*}
\begin{array}{ccccccc}
&         &0                                       &    &\mathbb F_p                           &&\\
&         &||                                     &    &||                                  &&\\
\cdots&\ra&H_r^e(M;V)&\ra &H_r(M;V)&\ra&\cdots\\
         &&\uparrow                               &    &||                        &   &\\
H_{r+1}^{cl}(F;V)&\ra&H_r^e(F;V)&\stackrel{\psi}\ra&H_r(F;V)&\ra&H_{r}^{cl}(F;V)\\
||&         &                                       &    &               &&||\\
H_{r+1+m-f}^{cl}(M;V)&&&   &      &&H_{r+m-f}^{cl}(M;V),         
\end{array}
\end{equation*}
whose rows are long exact homology sequences. 
The top row is computed using (\ref{end}) and the fact that 
$H_r(M;\mathbb F_p[\Gamma/\mathbb Z^r])=H_r(\widetilde M/\mathbb Z^r;\mathbb F_p)=\mathbb F_p.$
The vertical isomorphisms
\begin{eqnarray}
\label{smith1}
H_*(F;\mathbb F_p[\Gamma/\mathbb Z^r])&\cong& H_*(M;\mathbb F_p[\Gamma/\mathbb Z^r]),\\
\label{smith2}
H_{*+m-f}^{cl}(M;\mathbb F_p[\Gamma/\mathbb Z^r])&\cong&H_*^{cl}(F;\mathbb F_p[\Gamma/\mathbb Z^r]),
\end{eqnarray}
follow from Smith theory and will be proved below in Theorem \ref{coeffsmith}.
If $m>f$ then equation (\ref{closedzero}) shows that $\psi$ is an isomorphism.
This is a contradiction, since the commutative diagram shows that $\psi$ is the zero map. 
Thus, $m=f,$ which means the fixed point set is the entire manifold $M$, i.e. the homotopically trivial
$\mathbb Z/p$ action is actually trivial. To complete the proof of Theorem \ref{maintheorem},
we now deduce the Smith theory isomorphisms. 


\section{\label{Smith theory}Smith theory}
\subsection{\label{lifts}Lifted actions}

Let $f:M\ra M$ be a homeomorphism of the manifold $M$.
A {\it lift} of this homeomorphism to the universal cover $\widetilde M\ra M$ is a homeomorphism
$\widetilde f$ making the diagram

\begin{equation}
\begin{array}{ccc}
\widetilde M&\stackrel{\widetilde f}\longrightarrow &\widetilde M\\
\downarrow& &\downarrow\\
M&\stackrel{f}\longrightarrow &M
\end{array}
\end{equation}
commute. Composing $\widetilde f$ with a covering translation $\gamma\in\Gamma:=\pi_1M$ gives
another lift $\gamma\widetilde f$ of $f.$ If $G$ is a group acting by homeomorphisms on the manifold
$M$ then the lifts $L$ of these homeomorphisms to the universal cover form a group extension
\begin{equation*}
\label{liftsextension}
1\ra\Gamma\ra L\ra G\ra 1.
\end{equation*}
It is well known (see 8.8. in \cite{maclane}) that any group extension 
is determined by
\begin{enumerate}
\item
a representation 
$\rho:G\ra\Out(\Gamma),$ and
\item
a class in the cohomology group\footnote{The coefficients $Z(\Gamma)$ are a $G$-module via $\rho$.} $H^2(G;Z(\Gamma))$. 
\end{enumerate}
In the special case when the homomorphism $\rho$ and the center $Z(\Gamma)$ are both trivial
we get the trivial extension. Here is a short proof of this fact. 
\begin{prp}
\label{productlift}
If $1\ra \Gamma\ra L\ra G\ra 1$ is a group extension, $\Gamma$ has trivial center and the conjugation
homomorphism $\rho:G\ra\Out(\Gamma)$
is trivial, then the extension $L$ is isomorphic to the product $\Gamma\times G.$
\end{prp}
\begin{proof}
Since the homomorphism $\rho$ is trivial, for every $g\in L$ there is $\gamma(g)$ in $\Gamma$ such
that conjugation by $g$ agrees with conjugation by $\gamma(g)$ on $\Gamma.$ The element $\gamma(g)$ is 
unique because $\Gamma$ has trivial center, so we get a homomorphism $L\ra \Gamma, g\mapsto \gamma(g)$
which splits the extension. 
\end{proof}

\subsection{Lifted fixed point sets}
Since the locally symmetric space $M$ has no local torus factors,
its fundamental group is centerless, so Proposition \ref{productlift} implies the homotopically 
trivial $\mathbb Z/p$-action on $M$ lifts to a $\mathbb Z/p$-action on the universal cover which commutes with the 
action of the fundamental group $\Gamma:=\pi_1M$ by covering translations. 

\begin{lem}
\label{cover}
If $\phi\in\mathbb Z/p$ preserves the $\Gamma$-orbit of a point $x\in\widetilde M,$ then $\phi$ fixes the orbit pointwise.
\end{lem} 

\begin{proof}
If $\phi(\Gamma x)=\Gamma x,$ then there is 
$\gamma\in\Gamma$ such that 
$\phi(x)=\gamma x.$
Since $\phi$ and $\gamma$ commute and $\phi\in\mathbb Z/p$, we find

\begin{equation}
x=\phi^p(x)=\phi^{p-1}\phi(x)=\phi^{p-1}\gamma x=\gamma\phi^{p-1}(x)=\cdots=\gamma^px.
\end{equation}
Thus, the element $\gamma^p\in\Gamma$ fixes the point $x\in\widetilde M.$ 
Since the fundamental group acts freely on $\widetilde M$, we conclude $\gamma^p=1.$
The fundamental group of the aspherical manifold $M$ is torsionfree, so we must have $\gamma=1.$
Consequently, $\phi(x)=x$ and for any $\tau\in\Gamma$

\begin{equation}
\phi(\tau x)=\tau\phi(x)=\tau x,
\end{equation}  
so $\phi$ fixes the entire orbit $\Gamma x$ pointwise. 
\end{proof}

The lemma shows that the fixed point set $F$ is the projection
$\widetilde F/\Gamma$ of the fixed point set $\widetilde F$ of the $\mathbb Z/p$-action on the universal cover $\widetilde M.$ 
In other words, we have the commutative diagram 

\begin{equation}
\label{fixedsetsandcovers}
\begin{array}{ccc}
\widetilde F&\hookrightarrow&\widetilde{M}\\
\downarrow& &\downarrow\\
F&\hookrightarrow& M.
\end{array}
\end{equation}
The vertical maps are $\Gamma$-covers and the horizontal maps are inclusions. 
\subsection{Mod $p$ homology of fixed point sets}
In this subsection, we will prove the Smith theory isomorphisms (\ref{smith1}) and (\ref{smith2}).
This will complete the proof of Theorem \ref{maintheorem}.
\begin{thm}
\label{coeffsmith}
Let $M$ be an $m$-dimensional smooth aspherical manifold and $\widetilde M$ its universal cover.
Suppose we have a smooth $\mathbb Z/p$-action on $\widetilde M$ that commutes with the action
of the fundamental group $\Gamma:=\pi_1M$ by covering translations. Let $F\subset M$ be the fixed point set of the projected $\mathbb Z/p$-action
on $M$. Let $V$ be a $\mathbb F_p[\Gamma]$-module. Then, the inclusion of the fixed point set induces
isomorphisms on homology and cohomology with coefficients in $V$, i.e.

\begin{eqnarray}
\label{homology}
H_*(F;V)&\cong &H_*(M;V),\\
\label{cohomology}
H^*(M;V)&\cong &H^*(F;V).
\end{eqnarray}
Moreover, if the fixed point set has dimension $f$, then we have an isomorphism of homology with closed supports
\begin{equation}
\label{closed}
H_{*+m-f}^{cl}(M;V)\cong H_*^{cl}(F;V).
\end{equation}
\end{thm}
\begin{rem}
Geometrically, the isomorphism (\ref{closed}) is obtained by sending a cycle $c$ on $M$
to the transverse intersection $c\cap F$ on $F.$
\end{rem}
\begin{proof}
We pick a triangulation of the smooth manifold $M$ in which 
the smooth $\mathbb Z/p$-action is simplicial (this can be done by 7.1 in \cite{illman}.) 
The universal cover $\widetilde M$ is given the lifted (hence $\Gamma$-equivariant) triangulation. 
All chain complexes below are taken to be simplicial. 
Let $\widetilde F$ be the fixed point set of the 
$\mathbb Z/p$-action on the universal cover $\widetilde M.$
Lemma \ref{cover} above shows that this is the $\Gamma$-cover of $F.$ 

The key point is that 

\begin{equation}
\label{exact}
0\ra C_*(\widetilde F;\mathbb F_p)\ra C_*(\widetilde M;\mathbb F_p)\ra C_*(\widetilde M,\widetilde F;\mathbb F_p)\ra 0
\end{equation}
is an exact sequence of complexes of free $\mathbb F_p[\Gamma]$-modules, 
and the complex $C_*(\widetilde M,\widetilde F;\mathbb F_p)$ is acyclic.
\begin{itemize}
\item
The sequence is exact by definition.
\item
It is a sequence of complexes of $\mathbb F_p[\Gamma]$-modules because the $\mathbb Z/p$-action commutes with the $\Gamma$-action.
\item
For each $k$-simplex $\sigma_k$ in the base $M$, pick a lift $\widetilde\sigma_k$ in the universal cover $\widetilde M.$
Then, 
\begin{eqnarray*}
C_k(\widetilde F;\mathbb F_p)&\cong&\oplus_{\sigma_k\in F}\hspace{0.6cm}\mathbb F_p[\Gamma]\widetilde\sigma_k,\\
C_k(\widetilde M;\mathbb F_p)&\cong&\oplus_{\sigma_k}\hspace{1cm}\mathbb F_p[\Gamma]\widetilde\sigma_k,\\
C_k(\widetilde M,\widetilde F;\mathbb F_p)&\cong&\oplus_{\sigma_k\notin F}\hspace{0.6cm}\mathbb F_p[\Gamma]\widetilde\sigma_k,
\end{eqnarray*}
show that the $\mathbb F_p[\Gamma]$-modules in the sequence (\ref{exact}) are all free.
\item
The ordinary Smith theorem (VII.2.2 of \cite{bredon}) 
shows that the fixed point set $\widetilde F$ of a $\mathbb Z/p$-action on the
contractible manifold $\widetilde M$ has the $\mathbb F_p$-homology of a point, i.e. the
relative complex $C_*(\widetilde M,\widetilde F;\mathbb F_p)$ is acyclic. 
(Its homology vanishes in all dimensions.)
\end{itemize}
Recall that homology and cohomology with coefficients in $V$ are computed by\footnote{Here  $\widetilde{}$  denotes
the $\Gamma$-cover. This does not conflict with our earlier notation, because we've shown that in our situation 
the $\Gamma$-cover of the fixed point set is the fixed point set in the $\Gamma$-cover.}

\begin{eqnarray}
H_*(-;V)&=&H_*(C_*(\widetilde-;\mathbb F_p)\otimes_{\mathbb F_p[\Gamma]}V),\\
H^*(-;V)&=&H_*(\mathrm{Hom}_{\mathbb F_p[\Gamma]}(C_*(\widetilde-;\mathbb F_p),V)).
\end{eqnarray} 
Since the $\mathbb F_p[\Gamma]$-modules
appearing in the sequence (\ref{exact}) are all free, the 
functor $-\otimes_{\mathbb F_p[\Gamma]}V$ preserves exactness of the sequence (\ref{exact}) and acyclicity
of the relative complex $C_*(\widetilde M,\widetilde F;\mathbb F_p),$ so the long exact homology sequence
shows that $H_*(F;V)\ra H_*(M;V)$ is an isomorphism. The same remarks for the functor $\mathrm{Hom}_{\mathbb F_p[\Gamma]}(-,V)$
give the cohomology isomorphism (\ref{cohomology}). Finally, let $\mathcal O_M$
and $\mathcal O_F$ be the orientation $\mathbb F_p[\Gamma]$-modules on $M$ and $F$, respectively. If $p=2$,
then these modules are trivial, so they are equal. If $p$ is odd, then the normal bundle of the fixed point
set is orientable, so again the orientation modules are equal.  
Combining this with the Poincare duality isomorphisms
(Theorem V.9.2 of \cite{bredonsheaf})
\begin{eqnarray}
H^{cl}_*(M;V)&\cong &H^{m-*}(M;\mathcal O_M\otimes V),\\
H^{cl}_*(F;V)&\cong &H^{f-*}(F;\mathcal O_F\otimes V)
\end{eqnarray}
and the cohomology isomorphism (\ref{cohomology}) proves (\ref{closed}).
\end{proof}
\begin{rem}
The isomorphism (\ref{homology}) for $V=\mathbb F_p[\Gamma/\Lambda]$ is proved by Conner and Raymond
in the appendix of \cite{connerraymond}.
\end{rem}

\section{Proof of Theorem \ref{fw}}
The following proof uses methods developed by Farb and Weinberger in \cite{farbweinbergerisometries}. 
See also sections $7$ and $8$ of \cite{avramidi}. The main new ingredient is Theorem \ref{maintheorem}.
Let $\Gamma:=\pi_1M$ be the fundamental group.

\subsection*{Step 1: If $\Gamma$ commutes with a compact Lie group $K$, then $K=1$}
To show that $K$ is trivial, it suffices to show that it has no elements $\phi$ of prime order $p$.
Since $\phi$ commutes with $\Gamma$, it descends to a homotopically trivial $\mathbb Z/p$-action on the
locally symmetric space $M.$ We showed in the proof of Theorem \ref{maintheorem} that there are no
such actions. 

\subsection*{Step 2: If $\Gamma$ normalizes a compact Lie group $K$, then $K=1$}
We do this by successively eliminating characteristic compact subgroups.
Let $K_0<K$ be the identity component of the compact Lie group. We have the exact sequence
\begin{equation}
1\ra K_0^{sol}\ra K_0\ra K_0^{ss}\ra 1,
\end{equation}
where $K_0^{sol}$ is the maximal normal connected solvable subgroup and $K_0^{ss}$ is semi-simple. 
Since $K_0^{sol}$ is a compact connected solvable group, it is a torus $(S^1)^{n}.$ 
The finite subgroup $(\mathbb Z/p)^n<(S^1)^n$ is characteristic (it consists of the elements of order $1$ and $p$ in the torus),
$K_0^{sol}$ is characteristic in $K_0,$ the group $K_0$ is (topologically) characteristic in $K$ and $\Gamma$ acts on $K$ by conjugation, 
so $\Gamma$ also acts on the group $(\mathbb Z/p)^n$ by conjugation. Since this last group is finite, some finite index subgroup $\Gamma'<\Gamma$
acts trivially on it, i.e. $\Gamma'$ commutes with $(\mathbb Z/p)^n.$ By the claim above,
we have $n=0$ so that $K_0$ is semisimple. Thus, the center $Z(K_0)$ is finite so some finite index subgroup of $\Gamma$
commutes with it and by the claim we find that the center is trivial. Let $\Gamma'<\Gamma$ be the kernel of the conjugation homomorphism $\Gamma\ra\Out(K_0).$  Since $K_0$ is semisimple,
its outer automorphism group $\Out(K_0)$ is finite so the kernel is a finite index subgroup. By Proposition \ref{productlift},
the group generated by $K_0$ and $\Gamma'$ splits as a product $K_0\times\Gamma'$ and by the claim above
we conclude that $K_0$ is trivial. We've shown that $K$ is a finite group. Thus, there is a finite index subgroup $\Gamma'<\Gamma$
which commutes with $K$ and by the claim $K$ is trivial.

\subsection*{Step 3: Eliminating solvable and compact subgroups}
Let $I:=\Isom(\widetilde M,\widetilde g)$ be the isometry group and $I_0$ its identity component.
The group $I$ is a Lie group acting properly and smoothly on $\widetilde M.$ (\cite{myerssteenrod})
Let $\Gamma_0=\Gamma\cap I_0.$ 
Then, we have an exact sequence 
\begin{equation}
1\ra I_0^{sol}\ra I_0\ra I_0^{ss}\ra 1
\end{equation} 
where $I_0^{sol}$ is the maximal connected solvable normal subgroup of $I_0$ and $I_0^{ss}$ is semi-simple.
The group $\Gamma_0$ has \cite{raghunathan} a unique maximal normal solvable subgroup $\Gamma_0^{sol}.$
This subgroup is characteristic in $\Gamma_0$, so it is normal in $\Gamma.$ Since the locally symmetric
space has no local torus factors, the group $\Gamma$ has no non-trivial abelian normal subgroups 
(Theorem 10.3.10 of \cite{eberlein}), hence also no
non-trivial solvable normal subgroups, i.e. $\Gamma_0^{sol}=1.$ 
Farb and Weinberger \cite{farbweinbergerroyden} show that $\Gamma_0$ is a lattice in $I_0$. (This
only uses the assumption that $\widetilde g$ has finite $\Gamma$-covolume and
no particular properties of the mapping class group situation.)
This lets them use Prasad's Lemma $6$ in \cite{prasad} to conclude that $I_0^{sol}$ 
is compact and the center $Z(I_0^{ss})$ is finite. Since $I_0^{sol}$ is characteristic,
it is normalized by $\Gamma$. Since it is compact, Step 2 implies it is trivial. Thus, $I_0=I_0^{ss}$.
Now, $\Gamma$ acts by conjugation on the finite group $Z(I_0^{ss})$ so again Step 2
shows $Z(I_0^{ss})=1.$ Thus, $I_0$ is semisimple with trivial center.
Next, we note that it cannot have any compact factors: the product $K$ of the compact factors
is a characteristic subgroup, so it is normalized by $\Gamma$ and by Step 2 it must be trivial.
Thus, $I_0$ is semisimple with trivial center and no compact factors.

\subsection*{Step 4}
Now, look at the extension
\begin{equation}
1\ra I_0\ra\langle I_0, \Gamma \rangle\ra\Gamma/\Gamma_0\ra 1.
\end{equation}
Let $\Gamma'<\Gamma$ be the kernel of the conjugation homomorphism $\Gamma\ra\Out(I_0).$ 
It is a finite index subgroup of $\Gamma$, since $I_0$ is semisimple with trivial center. By
Proposition \ref{productlift} the extension 
$\langle I_0, \Gamma' \rangle$ splits as
$I_0\times (\Gamma'/\Gamma'_0),$ where $\Gamma'_0=\Gamma'\cap I_0.$ Consequently, 
$\Gamma'=\Gamma'_0\times (\Gamma'/\Gamma'_0)$. Since $\Gamma'$ is irreducible, we must have either
that $\Gamma'_0=1$ or that $\Gamma'_0$ is a finite index subgroup of $\Gamma'.$ Now we look at these two possibilities.

\subsection*{Case 1: $\Gamma'_0$ is a finite index subgroup of $\Gamma'$.}
Recall that $\Gamma_0$ is a lattice in the Lie group $I_0.$
Since $\Gamma'_0$ is a finite index subgroup of $\Gamma_0,$ it is also a lattice in $I_0.$
Let $K<I_0$ be a maximal compact subgroup. The quotient $I_0/K$ is a symmetric space
with no compact or Euclidean factors, so the locally symmetric space $\Gamma_0'\setminus I_0/K$ has the same
dimension as $M$. 
On the other hand, an orbit of $I_0$ acting on $\widetilde M$ has the form 
$I_0\cdot x=I_0/K_x$ for some compact subgroup $K_x<I_0.$
We have $\dim\widetilde M=\dim I_0/K\leq\dim I_0/K_x\leq\dim\widetilde M$ so the inequalities are actually
equalities. In other words, the isometry group $I_0$ is acting transitively on $\widetilde M$, i.e. 
$(\widetilde M,\widetilde g)$ is isometric to a symmetric space.  

\subsection*{Case 2: $\Gamma_0'$ is trivial}
The group $I_0$ is compact, since it contains $\Gamma_0'$ as a lattice. 
On the other hand, we've shown that $I_0$ has no compact factors, so $I_0=1$, i.e. the 
isometry group $I$ is discrete. Since $\Gamma$ is a lattice in $I$, we conclude that it is a finite index
subgroup of $I.$ Thus, there is a further finite index subgroup $\Gamma''<\Gamma$ which is normalized by $I$. 
Conjugation by $I$ gives a group homomorphism
$$
\rho:I/\Gamma''\ra\Out(\Gamma'').
$$
An element in the kernel of this homomorphism is a homotopically trivial isometry
of $(\widetilde M/\Gamma'',\widetilde g).$ By Theorem \ref{maintheorem}, there are no such homotopically
trivial isometries, so the map $\rho$ is injective. By Margulis-Mostow-Prasad rigidity, the image $\rho(I/\Gamma'')$ can be 
represented by isometries of the locally symmetric metric $h_{sym}$. Thus, the group $L$ of lifts
\begin{equation}
1\ra\Gamma''\ra L\ra\rho(I/\Gamma'')\ra 1
\end{equation}
ca be represented by isometries of the symmetric metric. This extension is determined by $\rho$ since $\Gamma''$ 
has trivial center (see subsection \ref{lifts}) so the group of lifts $L$ is isomorphic to $I$. 
We've shown that $I$ is a group of isometries of the symmetric metric. Let $\Gamma$ act
on the universal cover $\widetilde M$ by isometries via $\Gamma<I\cong L<\Isom(\widetilde M,\widetilde h_{sym}).$
By Margulis-Mostow-Prasad rigidity, the quotient $(\widetilde M/\Gamma,\widetilde h_{sym})$ is isometric to $(M,h_{sym})$. 
Consequently, 
\begin{equation}
|I/\Gamma|={\vol(\widetilde M/\Gamma,\widetilde h_{sym})\over\vol(\widetilde M/I,\widetilde h_{sym})}\leq{\vol(M,h_{sym})\over\varepsilon(h_{sym})},
\end{equation}
where $\varepsilon(h_{sym})$ is the volume of the smallest locally symmetric orbifold covered by $(M,h_{sym}).$
This completes the proof of Theorem \ref{fw}.
\begin{rem}
Kazhdan and Margulis have shown that the volume of the smallest locally
symmetric orbifold covered by a symmetric space with no compact of Euclidean factors $(\widetilde M,\widetilde h_{sym})$
is bounded below by a positive constant $\varepsilon(\widetilde h_{sym})$ 
that depends only on the symmetric space (c.f. XI.11.9 in \cite{raghunathan}).
By contrast, note that the real line covers circles of arbitrarily small volume.
\end{rem}

\bibliography{locsym}
\bibliographystyle{alpha}

\end{document}